\documentclass[12pt]{article}
\usepackage{amsmath,amssymb,amsfonts,amsthm,amscd,mathtools} 
\usepackage{mathrsfs} 
\usepackage{color}
\usepackage{enumerate}
\usepackage{authblk}
\usepackage{a4}
\usepackage{comment}
\providecommand{\keywords}[1]{Keywords: #1}
\def\S{\mathcal{S}}

\def\cL{\mathcal{L}}

\newcommand{\Pc}{{\mathcal{P}}}

\newcommand{\dd}{{\rm d}}
\newcommand{\WW}{\mathbb{W}}

\newcommand{\TT}{\mathbb{T}}

\newcommand{\iso}{\textrm{iso}}

\def\W{\mathcal{W}}
\def\div{\mathrm{div}}
\newcommand{\lcN}{\;\!\cal N}
\newcommand{\cN}{{\cal N}}
\newtheorem{theorem}{Theorem}[section]
\newtheorem{proposition}[theorem]{Proposition}

\newtheorem{lemma}[theorem]{Lemma}
\newtheorem{example}[theorem]{Example}

\theoremstyle{definition}

\begin{document}

\title{On decompositions of non-reversible processes}
\author[1]{M. H. Duong\footnote{h.duong@bham.ac.uk}}
\affil[1]{School of Mathematics, University of Birmingham, UK} 

\author[2]{J. Zimmer\footnote{jz@ma.tum.de}}
\affil[2]{Fakultät für Mathematik, Technische Universität München, Germany}

\maketitle

\begin{abstract}
  Markov chains are studied in a formulation involving forces and fluxes. First, the iso-dissipation force
  recently introduced in the physics literature is investigated; we show that its non-uniqueness is linked to
  different notions of duality giving rise to dual forces. We then study Hamiltonians associated to
  variational formulations of Markov processes, and develop different decompositions for them.
\end{abstract}

\section{Introduction}
Many processes in nature are directional in time, with diffusion being a simple but prominent example. In this article, we call these processes non-reversible, for reasons explained below. 
Further examples of non-reversible processes are drift-diffusion processes, 
\begin{equation}
  \label{eq:intr1}
  \partial_t\rho=-\mathrm{div}(f\rho)+\Delta\rho,
\end{equation}
where $\rho$ is a density and $f$ is a field. Here the nature of the field $f$ matters: if $f = \nabla V$ for some potential $V$, then the evolution can be understood in terms of an underlying free energy. If $f$ is not a gradient vector field, i.e., no such $V$ exists, then significant qualitative changes occur. These changes can have practical applications. For example, a system governed by  a free energy  converges under suitable conditions to equilibrium. Non-gradient terms change the rate of convergence, and can be used to accelerate the convergence (see, e.g.,~\cite{Kaiser2017} for an analysis on the level of Markov chains). This is sometimes visualised with milk being dropped in a cup of coffee: The milk will diffuse, and the mixture  which will eventually reach equilibrium (the homogeneous mixture of milk and coffee) without interference, but will reach equilibrium much faster if stirred by a spoon. 

This analogy can only lead so far. A more careful analysis reveals that processes can often be decomposed  in different components, which often satisfy suitable orthogonality relations. The understanding of such splits help to understand acceleration of convergence to equilibrium. For example, one can think of convergence to equilibrium as decreasing the free energy, thus undergoing a suitable gradient descent. If this takes place in a shallow part of the energy landscape, then the convergence will be slow. If there is an additional drift ``orthogonal'' to the gradient descent, then this can lead to steeper regions being reached, where convergence is faster. 

The situation gets even more complicated in processes out of equilibrium. One can think of a one-dimensional bar with both ends held at different systems. Then a cost is required to maintain the steady state, unlike the maintenance of the steady state in equilibrium. This cost is sometimes called housekeeping heat. The analysis of such phenomena is a very active field in stochastic thermodynamics. An example is recent work by Dechant, Sasa and Ito~\cite{Dechant2022a, Dechant2022b} on the geometric decomposition of entropy production in housekeeping, excess and coupling.

There are various levels on which non-reversible processes can be looked at -- on the continuum scale in form of a
partial differential equation as in~\eqref{eq:intr1}, or on the scale of underlying Markov processes. We will
largely focus on Markov processes. One reason is that their structure is remarkable in the following sense:
The evolution can be described in terms of forces $F$ and fluxes $j$, and the relation between them is always
the same (equation~\eqref{eq:j-F} below); the rates, which differentiate different Markov processes, enter
through the mobility (equation~\eqref{eq:a} below). The price to pay is that the force-flux relations on the
level of Markov chains are necessarily nonlinear. Section~\ref{sec:isosurface} summarises these classic results. The focus on Markov chains also explains the slightly
unusual terminology ``non-reversible'': This is to avoid the term ``irreversible'', which has a different
meaning for Markov processes.

\subsection{Outline of the paper}
The paper combines two thrusts of investigation.

In Section~\ref{sec:isosurface}, we first summarise some key notions for the description of  Markov processes in terms of forces and fluxes. We then explain the non-uniqueness of the so-called iso-dissipation force recently introduced in the physics literature in a mathematical way. 

In Section~\ref{sec:HJ-for-MC}, we use that the most likely evolution of a Markov chain can (in suitable
situations) be described by a variational principle, a so-called large deviation principle. There, the pathwise
evolution appears as minimiser of a functional. The functional is called rate functional and the integrand is
here denoted as Lagrangian. We consider its Legendre dual, thus a Hamiltonian. The central results of this section are various decompositions of the Hamiltonian. We also discuss an application of these splittings to diffusive
processes in terms of the Fokker-Planck equation.

\section{Orthogonality of forces and decomposition of the entropy production}
\label{sec:isosurface}

In this paper, we consider an ergodic Markov process $(X_t)_{t\geq 0}$ with generator $\mathcal{L}$ and a
unique invariant measure $\pi$. The time evolution of the probability density $\rho_t$ of $X_t$ can be written
in terms of the master equation (also called Fokker-Planck or forward Kolmogorov equation),
\begin{subequations}
\label{eq: general PDE}
\begin{align}
    \dot{\rho}_t&=\mathcal{L}'(\rho_t)\label{eq: dual-generator-form}
    \\&=-\div[j(\rho_t)]\label{eq: currentform}
    \\&=-\div[a(\rho_t)\phi(F(\rho_t))],\label{eq: force form}
\end{align}
\end{subequations}
where $\mathcal{L}'$ denotes the adjoint (with respect to the $L^2$ inner product) generator of $\mathcal{L}$,
$j$ is the flux (current), $a$ is the mobility and $F$ is the force. 

In~\eqref{eq: currentform} and~\eqref{eq: force form} we use a, possibly nonlinear, relation between forces
and fluxes
\begin{equation*}
  j(\rho)=a(\rho)\phi(F(\rho)),
\end{equation*}
for some function $\phi$. For instance, for the spatially continuous setting of diffusion processes, $j$ is
linearly dependent on $F$,
\begin{equation*}
  j(\rho)=a(\rho)F(\rho),
\end{equation*}
while for spatially discrete Markov chains, the relation between $j$ and $F$ is non-linear, and specifically
of the form
\begin{equation}
  \label{eq:j-F}
  j_{xy}(\rho)=a_{xy}(\rho)\sinh\Big(\frac{1}{2}F_{xy}(\rho)\Big),
\end{equation}
where
\begin{equation}
   \label{eq:a}
  a_{xy}(\rho)=2\sqrt{\rho(x)r_{xy}\rho(y)r_{yx}},\quad\text{and}\quad F_{xy}(\rho)=\log\frac{\rho(x)r_{xy}}{\rho(y)r_{yx}},
\end{equation}
with $r_{xy}$ being the transition rate for jumps from state $x$ to state $y$. This result goes back to
Schnakenberg~\cite{Schnakenberg1976a}, see also~\cite{Kaiser2018a}.

The formulations~\eqref{eq: dual-generator-form},~\eqref{eq: currentform} and~\eqref{eq: force form} give rise to different ways of decomposing a non-reversible dynamics into symmetric and anti-symmetric parts studied in recent years:
\begin{enumerate}[(i)]
\item A decomposition of the generator~\cite{GrmelaOttinger1997,DuongOttobre2021}: $\mathcal{L}=\mathcal{L}_S+\mathcal{L}_A$ (and thus of the dual operator
  $\mathcal{L}'$),
\item A decomposition of the fluxes~\cite{Kaiser2018a}: $j=j_S+j_A$,
\item A decomposition of the forces~\cite{Kaiser2018a,RengerZimmer2021,PattersonRengerSharma2022}: $F=F_S+F_A$.
\end{enumerate}
In this paper, we complement these decompositions by developing different splittings for Hamiltonians, that is, duals of Lagrangians appearing in a variational formulation of the Markov process. We now sketch this setting. 

Under fairly general conditions, see for example~\cite{Kaiser2018a}, the most probable state of a system is
described by a so-called large deviation principle (LDP), where the most probable state appears as a minimiser
of a rate functional. We consider this here in the situation where a Markov chain has a unique invariant
measure, and the stochastic process associated with the chain is carried out $\cN$ times (which van be interpreted as the
experiment described by the process being run $\cN$ times). Then the evolution of the system in terms of density $\rho_t$ an current $j_t$,
with subscript $t$ denoting time-dependence as common in stochastic processes, satisfies a \emph{large deviation
principle} (see, for example,~\cite{Kaiser2018a})
\begin{equation}
  \label{equ:pathwise-general}
  \mathrm{Prob}\left( (\hat\rho_t^{\lcN},\hat\jmath_t^{\lcN})_{t\in[0,T]} \approx (\rho_t, j_t)_{t\in[0,T]}\right)
  \asymp  \exp\left\{-\mathcal N I_{[0,T]}\left( (\rho_t, j_t)_{t\in[0,T]}\right)\right\},
\end{equation}
with \emph{rate function} of the form
\begin{equation}
  \label{eqn:mc_rate_functional}
  I_{[0,T]}\bigl((\rho_t, j_t)_{t\in[0,T]}\bigr)=  I_0(\rho_0) + \frac12\int_0^T \Phi(\rho_t,j_t, F(\rho_t))  \dd t ,
\end{equation}
where
\begin{equation}
  \label{eqn:Phi_function}
  \Phi(\rho,j,f):= \Psi(\rho,j) - \langle j, f\rangle + \Psi^*(\rho, f),
\end{equation} 
in which $\Psi$ is a convex functional, $\Psi^*$ is its Legendre dual, and $\langle j, f\rangle$ is a dual pairing
between a current $j$ and a force $f$. Here $\Psi$ and hence $\Psi^*$ depend on the process, in particular
on the rates. While $\Phi$ is a function of density $\rho$, flux $j$ and force $f$, the force in the physical
system is given as a function of $\rho$, which we write as uppercase symbol $F(\rho_t)$. For ease of notation, we will only write the second argument in the functionals $\Psi$ and $\Psi^*$ if no confusion can arise.

The dual pairing in the rate functional has a special meaning. Given a Markov process with flux $j$ and a
force $F$, the \emph{entropy production rate} associated with the process is given by
\begin{equation}
  e:=2\langle j,F \rangle,    
\end{equation}
where $\langle j,F \rangle $ is a dual pairing between the flux and the force as
in~\eqref{eqn:mc_rate_functional}.

It is of significant current interest in physics to understand the roles of entropy and entropy production, in
particular in processes out of equilibrium. For example, recently a connection between the response of a
physical observable and the relative entropy has been established~\cite{Dechant2020a}. A natural way to
understand entropy (production) is to split it in components. In~\cite{Kobayashi2022}, the authors introduce a
decomposition of the force as $F=F_S+F_A$ where
\begin{equation}
  \label{eq: decomposition of f}    
  F_S=\frac{F+F_{\iso}}{2},\quad\text{and}\quad F_A=\frac{F-F_{\iso}}{2}.
\end{equation}
In the formula above, $F_{\iso}$, which is called an \emph{iso-dissipation force}, is defined such that the identity 
\begin{equation}
  \label{eq: Fiso}
  \Psi^*(F)=\Psi^*(F_{\iso})   
\end{equation}
holds. Obviously this does not define $F_\iso$ uniquely, except for degenerate situations.

The decomposition of the force~\eqref{eq: decomposition of f} gives rise to a nonnegative decomposition of the
entropy production~\cite[Section II C]{Kobayashi2022}
\begin{align}
  e=2\langle j,f\rangle&=2\langle j, F_S\rangle+ 2\langle j, F_A\rangle    \notag
  \\&=D[j\|-j_{\iso}]+D[j\|j_{\iso}],\label{ed: decomposition of entropy production}
\end{align}
where $j_{\iso}$ is the Legendre transform of $F_{\iso}$ and given two fluxes $j_1, j_2$, $D[j_1\|j_2]$ is the \emph{Bregman divergence} between them, 
\begin{equation*}
  D[j_1\| j_2]:=\Psi(j_1)-\Psi(j_2)-\langle j_1-j_2,\partial_{j_1}\Psi(j_2)\rangle\geq 0.
\end{equation*}
In addition, according to~\cite{Kobayashi2022} there are infinitely many decompositions of the
force~\eqref{eq: decomposition of f} (and thus of the entropy production~\eqref{ed: decomposition of entropy
  production}), originating from different choices of iso-dissipation forces according to~\eqref{eq: Fiso}.

In the rest of this section, combining results from~\cite{Jack2017a} and an adaption of arguments
in~\cite{Kaiser2018a}, we provide a mathematical interpretation of iso-dissipation forces. The key idea is to
link them to the dual force (associated to a dual process of the original one).

As in~\cite{Jack2017a}, we define an adjoint process for which the probability of a path $(\rho_t,j_t)_{t\in[0,T]}$ is equal to the
probability of the time-reversed path $(\rho^*_t,j^*_t)_{t\in[0,T]}=(\rho_{T-t},-j_{T-t})_{t\in[0,T]}$ in the
original process. We assume that the adjoint process also satisfies an LDP of the
form~\eqref{equ:pathwise-general}, with rate function $I^*_{[0,T]}$. Hence we must have
\begin{equation} 
  \label{equ:assume-adjoint}
  I^*_{[0,T]}\bigl((\rho_t, j_t)_{t\in[0,T]}\bigr) = I_{[0,T]}\bigl((\rho^*_t, j^*_t)_{t\in[0,T]}\bigr) .
\end{equation}
Moreover, we assume that $I^*_{[0,T]}$ may be obtained from $I$ by replacing the force $F(\rho)$ with some
adjoint force $F^*(\rho)$ (this is discussed in~\cite{Kaiser2018a}. That is,
\begin{equation} 
  \label{equ:assume-Fstar}
  I^*_{[0,T]}\bigl((\rho_t, j_t)_{t\in[0,T]}\bigr)=  I_0(\rho_0) + \frac12\int_0^T \Phi(\rho_t,j_t, F^*(\rho_t))  \dd t.
\end{equation}
Here, $I_0$ is the rate function associated with fluctuations of the density $\rho$, for a system in its
steady state.  That is, within the steady state,
$\mathrm{Prob}(\hat\rho^{\;\!\cN}\approx \rho) \asymp \exp(-\cN I_0(\rho))$.

Then by~\cite[Proposition 3]{Kaiser2018a} we have
\begin{equation}
  \label{equ:FsFaOrth}
  \Psi^*\bigl(\rho,F(\rho)\bigr) = \Psi^*\bigl(\rho, F^*(\rho)\bigr) .
\end{equation}%
It follows from~\eqref{equ:FsFaOrth} and~\eqref{eq: Fiso} that one can identify an iso-dissipation force with
a dual force $F_{\iso}=F^*$.

We now show that mathematically there are infinitely many representations for the dual force $F^*$. Therefore,
there are indeed infinitely choices for an iso-dissipation force.


We define a time-reversal operation $\TT_0$ which reverses time but does not change any coordinates or
momenta. That is, for paths $X$ on the time interval $[0,\tau]$, we take $(\TT_0X)_t=(X)_{\tau-t}$. Now define
an adjoint dynamics~\cite{Bertini2015} for which the path measure is $\Pc^*$, with
\begin{equation}
  \dd\Pc^*(X)=\dd\Pc\left(\TT_0 X\right) .
  \label{equ:P*}
\end{equation}
Let $\WW^*$ be the corresponding operator for the adjoint process. Then $\WW^*$ satisfies
\begin{equation}
  \label{eq1}
  (\WW)_{y,x}\pi(x)=(\WW^*)_{x,y}\pi(y).
\end{equation}
Let $\mu$ be a measure which is absolutely continuous with respect to the invariant measure $\pi$. We denote
by $h=\frac{\dd\mu}{\dd\pi}$ the corresponding Radon-Nikodym derivative. The following argument is adapted
from~\cite{Jack2017a}. Let $\WW^+_\mu$ be the adjoint of $\WW$ with respect to $\mu^{-1}$, that is
\begin{equation*} 
  \int f(y)(\WW g)(y)\mu^{-1}(y)\dd y=\int g(x)(\WW^+_\mu f)(x)\mu^{-1}(x)\dd x.
\end{equation*}
The above equality can be rewritten as
\begin{equation*}
  \int\int f(y)(\WW)_{y,x}g(x)\mu^{-1}(y)\dd x\dd y=\int\int g(x)(\WW^+_\mu)_{x,y} f(y)\mu^{-1}(x)\dd x\dd y,
\end{equation*}
which implies that
\begin{equation}
  \label{eq2}
  (\WW)_{y,x}\mu^{-1}(y)=(\WW^+_\mu)_{x,y}\mu^{-1}(x).
\end{equation}
From~\eqref{eq1} and~\eqref{eq2} we deduce that
\begin{equation*}
  (\WW^*)_{x,y}=(\WW)_{y,x}\frac{\pi(x)}{\pi(y)}=\frac{\pi(x)}{\mu(x)}(\WW^+_\mu)_{x,y}\frac{\mu(y)}{\pi(y)}.
\end{equation*}
Thus
\begin{equation}
  \label{eq: representation of dual operator}    
  \WW^*=\WW^*_\mu=h^{-1}\circ \WW^+_\mu \circ h.
\end{equation}
Hence for each $\mu$,~\eqref{eq: representation of dual operator} provides a representation of $\WW^*$, which
in turn gives rise to a dual force $F^*=F^*_\mu$. As a consequence, using the identification between an
iso-dissipation force and a dual force discussed after~\eqref{equ:FsFaOrth}, it follows that there are
infinitely many iso-dissipation forces (and thus infinitely many different ways of decomposing the entropy
production).

\begin{example}
  Consider the following non-reversible diffusion process
  \begin{equation}
    \label{eq: nonreversible SDE}    
    \dd X_t=f(X_t)\,\dd t+\sqrt{2}\sigma\,\dd W(t).
  \end{equation}
  The associated Fokker-Planck equation is
  \begin{equation*}
    \partial_t\rho=\WW\rho=-\mathrm{div}(f\rho)+D \Delta\rho=-\mathrm{div}(\rho F(\rho)),
  \end{equation*}
  where $D=\sigma^2$ and
  \begin{equation*}
    F(\rho)=f-D \log\nabla\rho.
  \end{equation*}
  Suppose the invariant measure is of the form $\dd\pi(x)\propto e^{-U(x)}\dd x$ (this is a common assumption
  in physics, with the difficulty that $U$ is in general not known).  Taking $\mu(dx)=\dd x$, the
  corresponding adjoint process is
  \begin{equation*}
    \partial_t \rho = \WW^* \rho = -\mathrm{div}(f^*\rho)+D \Delta\rho=-\mathrm{div}(\rho F^*(\rho)),
  \end{equation*}
  where
  \begin{equation*}
    f^*(\rho)=-(2D \nabla U+f), \quad\text{and}\quad F^*(\rho)=f^*(\rho)-D\log\nabla\rho.
  \end{equation*}
  Thus a decomposition of the force is
  \begin{align*}
    F_S(\rho)=\frac{F(\rho)+F^*(\rho)}{2}& =-D\nabla U-D\log\nabla\rho \intertext{and}
                                           F_A(\rho)&
                                                      =\frac{F(\rho)-F^*(\rho)}{2}=f+D\nabla U.
  \end{align*}
  While $F_S$ and $F_A$ are uniquely defined, there are infinitely many choices of $F_\iso$ giving rise to this decomposition, via~\eqref{eq: decomposition of f}. 
\end{example}

\section{Hamilton-Jacobi formulations for Markov chains}
\label{sec:HJ-for-MC}

In this section we review the Hamiltonian viewpoint of the picture described above, and introduce splittings on the Hamiltonian level. As discussed in~\cite[Eq. (44)]{Kaiser2018a}, the Hamiltonian associated with Markov chains is
\begin{equation}
   \mathscr{H}(\rho,\xi) = \frac12\left[ \Psi^*(\rho,F(\rho) + 2\xi) -   \Psi^*(\rho,F(\rho)) \right].
  \label{equ:ham-markov}
\end{equation}
Let $ \mathscr{L}$ be the Legendre dual of $ \mathscr{H}$, that is
\begin{equation*}
 \mathscr{L}(\rho,j)=\sup_{\xi}\{\langle j,\xi\rangle - \mathscr{H}(\rho, \xi)\}.
\end{equation*}
Given the Hamiltonian $\mathscr{H}$ and the force $F$, we can also find the functional $\Psi^*$ from \eqref{equ:ham-markov} by
\begin{equation}
\Psi^*(\rho, \xi)=2\Big[\mathscr{H}(\rho,\frac{1}{2}(\xi-F(\rho)))-\mathscr{H}(\rho,-\frac{1}{2}F(\rho))\Big].    
\end{equation}

\begin{lemma}
The Hamiltonian being of the form~\eqref{equ:ham-markov} is equivalent to the Lagrangian being given by 
\begin{equation}
\label{eq:lang-markov}
     \mathscr{L}(\rho,j)=\frac{1}{2}\Big[\Psi(\rho, j)-2\langle j,F(\rho)\rangle+\Psi^*(\rho, F(\rho))\Big].
\end{equation}
\end{lemma}
We note that~\eqref{eq:lang-markov} is precisely the form of the integrand of a rate functional we have
encountered in~\eqref{eqn:Phi_function}, with the middle term being the entropy production rate. 
\begin{proof}
This proof is adapted from \cite{Mielke2014} to be consistent with our definition \eqref{equ:ham-markov} (in \cite{Mielke2014} the authors use a slightly different definition where there are no factors $1/2$ and $2$ on the right-hand side of \eqref{equ:ham-markov}). Suppose $\mathscr{H}$ is given as in \eqref{equ:ham-markov}. Then we get
\begin{align*}
 \mathscr{L}(\rho,j)&=\sup_{\xi}\{\langle j,\xi \rangle - \mathscr{H}(\rho, \xi)\}
\\&=\sup_{\xi}\{\langle j,\xi\rangle -\frac{1}{2}\Psi^*(\rho, F(\rho)+2\xi)+\frac{1}{2}\Psi^*(\rho, F(\rho))\}
\\&=\sup_{\xi}\{\langle j,\xi\rangle -\frac{1}{2}\Psi^*(\rho, F(\rho)+2\xi)\}+\frac{1}{2}\Psi^*(\rho, F(\rho))
\\&=\frac{1}{2}\sup_{\xi}\{\langle j,2\xi\rangle -\Psi^*(\rho, F(\rho)+2\xi)\}+\frac{1}{2}\Psi^*(\rho, F(\rho))
\\&=\frac{1}{2}\sup_{\xi}\{\langle j,\xi\rangle -\Psi^*(\rho, F(\rho)+\xi)\}+\frac{1}{2}\Psi^*(\rho, F(\rho))
\\&=\frac{1}{2}\sup_{\xi}\{\langle j,\xi-F(\rho)\rangle -\Psi^*(\rho, \xi)\}+\frac{1}{2}\Psi^*(\rho, F(\rho))
\\&=\frac{1}{2}\sup_{\xi}\{\langle j,\xi\rangle -\Psi^*(\rho, \xi)\}-\langle j,F(\rho)\rangle+\frac{1}{2}\Psi^*(\rho, F(\rho))
\\&=\frac{1}{2}\Psi(\rho, j)-\langle j,F(\rho)\rangle+\frac{1}{2}\Psi^*(\rho, F(\rho)).
\end{align*}
Similarly suppose $ \mathscr{L}$ is given in \eqref{eq:lang-markov} then one obtains \eqref{equ:ham-markov}.  
\end{proof} 


\subsection{Splittings of Hamiltonians}
\label{sec:splitt-hamilt}

The question addressed in this section is: Given an evolution of a Markov process governed by a Lagrangian of the form~\eqref{eq:lang-markov}, can this variational formulation be split in symmetric and asymmetric parts? For example,  assume we consider a process with a given rate functional~\eqref{eq:lang-markov} and construct a second process with the same minimiser of the rate functional, where the rate functional is equal or higher than the one of the first process. This can be interpreted as acceleration of convergence to equilibrium: The equilibrium (minimiser) has not changed, but the steeper nature of the second functional can lead to faster convergence. For example, it is illuminating if a contribution due to an asymmetric term can be shown to increase the Lagrangian. If the functionals were quadratic, polarisation identities could be easily used. The situation with nonquadratic functionals $\Psi^\star$ is more complex (also on the level of the associated evolution equation, which is then nonlinear). We develop here splittings in the dual (Hamiltonian) picture, where information of this kind can be read off. While the splittings appear technical at first sight, we illustrate them with an example in Subsection~\ref{sec:appl-diff-proc}. 


We consider a general decomposition of the Hamiltonian $\mathscr{H}$ as a sum of two Hamiltonians and study its implications. The computations are formal, as we assume sufficient regularity and convexity conditions to ensure the existence of derivatives of the relevant functionals and to guarantee the existence of maximisers (minimisers) in the relevant suprema (infima).  We first characterise the minimiser of $\mathscr{H}$. The following lemma is elementary and well-known. We include the proof for completeness.

\begin{lemma}
  \label{lem: minH}
  Let $\mathscr{L},\mathscr{H}$ be Legendre duals. Then
  $ \min\limits_{\xi}\mathscr{H}(\rho,\xi)=-\mathscr{L}(\rho,0)$ is achieved at
  $\xi^*=\partial_j\mathscr{L}(\rho,0)$.
\end{lemma}

\begin{proof}
By definition of $\mathscr{H}$, we have
\begin{equation}
  \label{eq: H interms of s_xi}
  \mathscr{H}(\rho,\xi)=\langle\xi,s_\xi\rangle-\mathscr{L}(\rho,s_\xi),
\end{equation}
where $s_\xi$ satisfies
\begin{equation}
  \label{eq: s_xi}
  \xi=\partial_j\mathscr{L}(\rho,s_\xi).
\end{equation}
The optimal $\xi_{\mathrm{opt}}$ is thus found through
\begin{equation}
  \label{eq: xi_opt 1}
  \xi_{\mathrm{opt}}=\partial_j\mathscr{L}(\rho,s_{\xi_{\mathrm{opt}}}),
\end{equation}
and
\begin{align*}
  \label{eq: xi_opt 1}
  0&=\partial_\xi\mathscr{H}(\rho,\xi_{\mathrm{opt}})
  \\&=s_{\xi_{\mathrm{opt}}}+\xi_{\mathrm{opt}}\partial_{\xi_{\mathrm{opt}}}
  s_{\xi_{\mathrm{opt}}}-\partial_j\mathscr{L}(\rho,s_{\xi_{\mathrm{opt}}})
  \partial_{\xi_{\mathrm{opt}}}s_{\xi_{\mathrm{opt}}}
  \\&=s_{\xi_{\mathrm{opt}}}+(\xi_{\mathrm{opt}}-\partial_j\mathscr{L}(z,s_{\xi_{\mathrm{opt}}}))
  \partial_{\xi_{\mathrm{opt}}}s_{\xi_{\mathrm{opt}}}
  \\&=s_{\xi_{\mathrm{opt}}}.
\end{align*}
Hence $\xi_{\mathrm{opt}}=\partial_j\mathscr{L}(\rho,0)$ and
$\min\limits_{\xi}\mathscr{H}(\rho,\xi)=\mathscr{H}(\rho,\xi_{\mathrm{opt}})=-\mathscr{L}(\rho,0)$ as claimed.
\end{proof}
We now give different decompositions of Hamiltonian $\mathscr{H}$.  The first result provides a decomposition
of the Hamiltonian $\mathscr{H}=\mathscr{H}_1+\mathscr{H}_2$. The key idea is to relate the Legendre duality
of the sum of two operators with the duality of each of them, which was studied for instance
in~\cite{Attouch1996}. We will apply this general result to the cases where $\mathscr{H}_1$ and
$\mathscr{H}_2$ are constructed respectively from the reversible and irreversible parts of the underlying
process.

\begin{proposition}
  \label{prop: general decomposition of H}
  Suppose that the Hamiltonian $\mathscr{H}$ can be decomposed as a sum of two Hamiltonians $\mathscr{H}_1$
  and $\mathscr{H}_2$,
  \begin{equation}
    \label{eq: general decomposition of H}
    \mathscr{H}(\rho,\xi)=\mathscr{H}_1(\rho,\xi)+\mathscr{H}_2(\rho,\xi),
  \end{equation}
  where $\mathscr{H}_1(\rho,0)=\mathscr{H}_2(\rho,0)=0$.  Let $\mathscr{L}_1$, $\mathscr{L}_2$ and $\mathscr{L}$
  be the Legendre transformation of $\mathscr{H}_1$, $\mathscr{H}_2$ and $\mathscr{H}$. Then we have
  \begin{equation}
    \label{eq: L vs L1 and L2}
    \mathscr{L}(\rho,j)=\mathscr{L}_1(\rho,\partial_\xi \mathscr{H}_1(\rho,\xi'))
    +\mathscr{L}_2(\rho,\partial_\xi \mathscr{H}_2(\rho,\xi')),
  \end{equation}
  where each term on the right-hand side is non-negative and $\xi'$ satisfies
  \begin{equation}
    \label{eq: eq xi'}
    \partial_\xi \mathscr{H}(\rho,\xi')=j.    
  \end{equation}
  As a consequence, we have the following decomposition 
  \begin{equation}
    \label{eq: decomposition of I general}
    \mathscr{L}(\rho,0)=\mathscr{L}_1(\rho,
    \partial_{\xi}\mathscr{H}_1(\rho,\xi^*))+\mathscr{L}_2(\rho,\partial_{\xi}\mathscr{H}_2(\rho,\xi^*)),
  \end{equation}
  where $\xi^*=\partial_j \mathscr{L}(\rho,0)$, which is a solution of
  \begin{equation}
    \label{eq: eq xi*}    
    \partial_{\xi}\mathscr{H}(\rho,\xi^*)=0.
  \end{equation}
\end{proposition}

Note that $ \mathscr{L}(\rho,0)$ 
depends implicitly on $\xi^*$, which solves~\eqref{eq: eq xi*}, via $\partial_\xi\mathscr{H}_1(\rho,\xi^*)$ and
$\partial_\xi\mathscr{H}_2(\rho,\xi^*)$.

\begin{proof}
We express $\mathscr{L}$ in terms of $\mathscr{L}_1$ and $\mathscr{L}_2$. We have
\begin{align}
  \label{eq: eqL1}
  \mathscr{L}(\rho,j)&=\sup_{\xi}\{\langle j,\xi\rangle- \mathscr{H}(\rho,\xi)\}\notag
  \\&=\sup_{\xi}\{\langle j,\xi\rangle- \mathscr{H}_1(\rho,\xi)-\mathscr{H}_2(\rho,\xi)\}\notag
  \\&= \langle j,\xi'\rangle- \mathscr{H}_1(\rho,\xi')-\mathscr{H}_2(\rho,\xi'),
\end{align}
where $\xi'$ solves for given $j$
\begin{equation*}
  j=\partial_{\xi}\mathscr{H}_1(\rho,\xi')+\partial_{\xi}\mathscr{H}_2(\rho,\xi')=\partial_{\xi}\mathscr{H}(\rho,\xi').
\end{equation*}
Let $s_1:=\partial_{\xi}\mathscr{H}_1(\rho,\xi')$, then $j-s_1=\partial_{\xi}\mathscr{H}_2(\rho,\xi')$. Therefore
\begin{equation}
  \label{eq: eqL2}
  \mathscr{L}_1(\rho,s_1)=\langle s_1,\xi'\rangle-\mathscr{H}_1(\rho,\xi'),\quad \mathscr{L}_2(\rho,j-s_1)=\langle j-s_1,\xi'\rangle-\mathscr{H}_2(\rho,\xi').  
\end{equation}
From~\eqref{eq: eqL1} and~\eqref{eq: eqL2}, it follows that
\begin{equation*}
  \mathscr{L}(\rho,j)=\mathscr{L}_1(\rho,s_1)+\mathscr{L}_2(\rho,j-s_1)=\mathscr{L}_1(\rho,\partial_\xi \mathscr{H}_1(\rho,\xi'))+\mathscr{L}_2(\rho,\partial_\xi \mathscr{H}_2(\rho,\xi')),
\end{equation*}
which is~\eqref{eq: L vs L1 and L2}. In addition, since $\mathscr{H}_1(\rho,0)=\mathscr{H}_2(\rho,0)=0$ we have
\begin{equation*}
  \mathscr{L}_1(\rho,j)=\sup\{\langle j,\xi\rangle-\mathscr{H}_1(\rho,\xi)\}\geq \langle 0,s\rangle-\mathscr{H}_1(\rho,0)=0.
\end{equation*}
Similarly $\mathscr{L}_2(\rho,j)\geq 0$. Hence each term on the right-hand side of~\eqref{eq: L vs L1 and L2}
is non-negative.

Applying~\eqref{eq: L vs L1 and L2} to $j=0$, 
we obtain~\eqref{eq: decomposition of I general}, where $\xi^*$ solves the equation
\begin{equation*}
  \partial_\xi \mathscr{H}(\rho,\xi^*)=0,
\end{equation*}
and according to Lemma~\ref{lem: minH}, we get $\xi^*=\partial_j \mathscr{L}(\rho,0)$.
\end{proof}

Next we consider a number of applications of Proposition~\ref{prop: general decomposition of H}. In particular, we provide different decompositions of  $\mathscr{L}(\rho,0)=\sup_{\xi}(-\mathscr{H}(\rho,\xi))$, which is of particular interest since it gives the Donsker-Varadhan rate functional for the empirical occupation  measures of the Markov process~\cite{DonskerVaradhan1975a}.
 The first
application is a characterisation of the (Donsker-Varadhan) rate functional of reversible processes. Here we follow the definition of reversibility in \cite{Kraaij2020}, that is, given a functional $\S\in C^1$, we say that the Hamiltonian $\mathscr{H}$ is \emph{reversible} with respect to $\S$ if $\mathscr{H}(\rho,\xi)=\mathscr{H}(\rho,d\S(\rho)-\xi)$ for all
  $(\rho,\xi)$. In particular, when $\S$ is the relative entropy, then this is equivalent to the usual detailed balance condition and time-reversibility \cite{Mielke2014}. According to \cite{Kraaij2020},  the Legendre pair of functionals $(\Psi, \Psi^*)$ associated to $\mathscr{H}$ as in \eqref{equ:ham-markov} with $F(\rho)=d\S(\rho)$ are strictly convex, continuously differentiable, symmetric (in the second argument) and satisfy $\Psi(\rho,0)=\Psi^*(\rho,0)=0$  if and only if $\mathscr{H}$ is reversible with respect to $\S$.


\begin{lemma}
\label{lem: reversible}
  Suppose that $\mathscr{H}(\rho,\xi)=\mathscr{H}_2(\rho,\xi)$, where $\mathscr{H}_2$ is reversible with respect
  to some functional $\S$. Let $\mathscr{L}_2$ be the associated Lagrangian and $\Psi^*_2$ be the dissipation
  potential associated to $\mathscr{H}_2$ and $d\S$, that is 
  \begin{equation}
  \label{eq:psi2}
    \Psi_2^*(\rho,\xi)=2\Big[\mathscr{H}_2(\rho, \frac1{2}(\xi-d\S(\rho))-\mathscr{H}_2(\rho, -\frac{1}{2}d\S(\rho))\Big].
  \end{equation}
  Then we have
  \begin{equation*}
    \mathscr{L}_2(\rho,0)=\frac{1}{2}\Psi_2^*(\rho,d\S(\rho)).    
  \end{equation*}
\end{lemma}

\begin{proof}
Since $\mathscr{H}=\mathscr{H}_2$ is reversible with respect to $\S$, we have
\begin{equation*}
  \mathscr{H}(\rho, \frac{1}{2}d\S-\xi)=\mathscr{H}(\rho, \xi+\frac{1}{2}d\S) \qquad\forall (\rho,\xi).
\end{equation*}
Taking the derivative with respect to $\xi$ on both sides yields
\begin{equation*}
  \partial_\xi\mathscr{H}(\rho,\frac{1}{2}d\S-\xi)
  =-\partial_\xi\mathscr{H}(\rho, \xi+\frac{1}{2}d\S) \qquad\forall (\rho,\xi).
\end{equation*}
This implies that
\begin{equation*}
  \partial_\xi\mathscr{H}(\rho,\frac{1}{2}d\S)=0.
\end{equation*}
Thus $\xi^*=\frac{1}{2}d\S$ is a solution to~\eqref{eq: eq xi*}. Therefore,
\begin{equation*}
  \mathscr{L}_2(\rho,0)=   \frac{1}{2} \Psi_2^*(\rho,d\S(\rho)),
\end{equation*}
where the last equality is~\eqref{eq:lang-markov} with $j=0$.
\end{proof}

Next, we consider a special case where $\mathscr{H}_2$ is independent of the second argument.
\begin{lemma}
  \label{lem: H2 independent of xi}
  Suppose that 
  \begin{equation}
    \label{eq: H2 independent of xi}
    \mathscr{H}(\rho,\xi)=\mathscr{H}_1(\rho,\xi)+\mathscr{H}_2(\rho).
  \end{equation}
  Then the associated Lagrangian $\mathscr{L}$ is given by 
  \begin{equation}
    \mathscr{L}(\rho,0) =\mathscr{L}_1(\rho,0)-\mathscr{H}_2(\rho).
  \end{equation}
\end{lemma}

\begin{proof}
  Since $\mathscr{H}_2(\rho,\xi)=\mathscr{H}_2(\rho)$, $\mathscr{L}_2$ is defined only on
  $\mathrm{Dom}(\mathscr{L}_2)=\{0\}$ and $\mathscr{L}_2(\rho,0)=-\mathscr{H}_2(\rho)$. Since
  $\partial_\xi\mathscr{H}_2(\rho,\xi)=0$, it follows that
  $\partial_\xi\mathscr{H}(\rho,\xi^*)=\partial_\xi\mathscr{H}_1(\rho,\xi^*)=0$. Thus
  \begin{align*}    
    \mathscr{L}(\rho,0) &=\mathscr{L}_1(\rho,\partial_{\xi}\mathscr{H}_1(\rho,\xi^*))+\mathscr{L}_2(\rho,\partial_{\xi}\mathscr{H}_2(\rho,\xi^*))
                            =\mathscr{L}_1(\rho,0)+\mathscr{L}_2(\rho,0)
    \\&=\mathscr{L}_1(\rho,0)-\mathscr{H}_2(\rho).
  \end{align*}
  Alternatively, this can be seen directly by
  \begin{align*}    
    \mathscr{L}(\rho,0) &=\sup_{\xi}\{-\mathscr{H}(\rho,\xi)\}=\sup_{\xi}\{-\mathscr{H}_1(\rho,\xi)-\mathscr{H}_2(\rho)\}=\sup_{\xi}\{-\mathscr{H}_1(\rho,\xi)\}-\mathscr{H}_2(\rho)
    \\&=\mathscr{L}_1(\rho,0)-\mathscr{H}_2(\rho).
  \end{align*}
\end{proof}
Next, we consider a special case where $\mathscr{H}_1$ is linear with respect to the second argument and
$\mathscr{H}_2$ is reversible. In this case, we will be able to determine
$\partial_\xi \mathscr{H}_1(\rho,\xi^*)$ and $\partial_\xi\mathscr{H}_2(\rho,\xi^*)$ in~\eqref{eq: decomposition
  of I general} explicitly.
\begin{lemma}
  \label{lem: H1 linear}
  Suppose that 
  \begin{equation}
    \label{eq: H1 linear}
    \mathscr{H}(\rho,\xi)=\langle \W(\rho),\xi\rangle+\mathscr{H}_2(\rho,\xi),
  \end{equation}
  where $\mathscr{H}_2(\rho,\xi)$ is symmetric around $d\S$. Let $\mathscr{L}_2(\rho,s)$ be the Legendre dual of
  $\mathscr{H}_2(\rho,\xi)$ and $\Psi_2^*$ be the dissipation potential defined from $\mathscr{H}_2$ as
  in~\eqref{eq:psi2}.  Then we have
  \begin{equation}
    \mathscr{L}(\rho,0)
    =\mathscr{L}_2(\rho,-\W(\rho))= \frac{1}{2}\Psi_2(\rho,-\W(\rho))+\frac{1}{2}\Psi_2^*(\rho,d\S(\rho))+\langle \W(\rho),d\S(\rho)\rangle.
  \end{equation}
  In particular, if $\langle d\S(\rho),\W(\rho)\rangle=0$, then
  \begin{equation*}
    \mathscr{L}(\rho,0)
    =  \Psi_2(\rho,-\W(\rho))+\Psi_2^*(\rho,-\frac{1}{2}d\S(\rho)).
  \end{equation*}
\end{lemma}

Below, we apply this lemma to obtain a decomposition of the so-called Donsker-Varadhan rate functional for non-reversible diffusion
processes, see Subsection~\ref{sec:appl-diff-proc}. 

\begin{proof}
  This lemma is an application of Proposition~\ref{prop: general decomposition of H}, where
  $\mathscr{H}_1(\rho,\xi)=\langle \W(\rho),\xi\rangle$, thus $\mathscr{L}_1(\rho,\cdot)$ is defined only on
  $\mathrm{Dom}(\mathscr{L}_1)=\{\W(\rho)\}$ and $\mathscr{L}_1(\rho,\W(\rho))=0$. In addition, since
  $\partial_{\xi}\mathscr{H}_1(\rho,\xi^*)=\W(\rho)$, we have
  $\partial_{\xi}\mathscr{H}_2(\rho,\xi^*)=-\W(\rho)$. Thus
  \begin{align*}
    \mathscr{L}(\rho,0) 
    &=\mathscr{L}_1(\rho,\W(\rho))+\mathscr{L}_2(\rho,-\W(\rho))=\mathscr{L}_2(\rho,-\W(\rho))
    \\&=\frac{1}{2}\Big[\Psi_2(\rho,-\W(\rho))+\Psi_2^*(\rho,d\S(\rho))+2\langle \W(\rho),d\S(\rho)\rangle\Big],
  \end{align*}
  where the last equality is~\eqref{eq:lang-markov} with $j=-\W(\rho)$ and $F(\rho)=d\S(\rho)$.
  We can also prove this directly as follows.
  \begin{align*}
    \mathscr{L}(\rho,0) 
    =\sup_{\xi}\{-\mathscr{H}(\rho;\xi)\}&=\sup_{\xi}\{-\langle \W(\rho),\xi\rangle-\mathscr{H}_2(\rho,\xi)\}\\&=\mathscr{L}_2(\rho,-\W(\rho)),
  \end{align*}
   where  
   the second  equation is by~\eqref{eq: H1 linear}.
\end{proof}

\subsection{Application to diffusion processes}
\label{sec:appl-diff-proc}

As application, we consider general non-reversible diffusion process of the form~\eqref{eq: nonreversible SDE}, where the diffusion matrix may depend on the position,   
\begin{equation*}
    \dd X_t=b(X_t)\,\dd t+\sqrt{2}\sigma(X_t)\,\dd W(t).
\end{equation*}
Let $\mu$ be the invariant measure of the process and $\S_\mu(\rho)$ be the relative entropy between $\rho$
and $\mu$, and $D=\sigma \sigma^T$. Then, one has~\cite{Mielke2014, DuongOttobre2021} for the generator
\begin{align*}
\mathcal{L}_s\phi&=\div(D\nabla\phi)+D\nabla\phi\cdot\nabla\log\mu,
\\ \cL_a\phi&=\cL\phi-\cL_s\phi=b\cdot\nabla\phi-D\nabla\phi\cdot\nabla\log\mu,
\\\cL'_s\rho&=\div(D\nabla\rho)-\div(\rho D\nabla\log\mu)=\div\left[\rho D\nabla\big(\log(\rho/\mu)\big)\right]=\div\left[\rho D\nabla \big(\dd\S_\mu(\rho)\big)\right],
\\ \cL_a' \rho&=\div(\rho D\nabla\log\mu-b\rho),
\end{align*}
and furthermore
\begin{align*}
  \mathscr{H}(\rho;\xi)&=\int e^{-\xi}\cL e^{\xi}\rho=(\xi, \cL'\rho)+( D\nabla\xi\cdot\nabla\xi,\rho).
\intertext{The symmetric Hamiltonian is given by}
 \mathscr{H}_s(\rho;\xi)&=\int e^{-\xi}\cL_s e^{\xi}\rho=(\xi, \cL_s'\rho)+(D\nabla\xi\cdot\nabla\xi,\rho).
            \intertext{The dissipation potential associated to the symmetric Hamiltonian (with $F(\rho)=-\dd\S_{\mu}(\rho)$) is given by} \Psi_s^*(\rho;\xi)&=2\left[\mathscr{H}_s\Big(\rho;\frac{1}{2}(\xi+\dd\S_\mu(\rho))\Big)-\mathscr{H}_s\Big(\rho;\frac{1}{2}d\S_\mu(\rho)\Big)\right]
                         \\&=\frac{1}{2}(D\nabla \xi\cdot\nabla \xi,\rho).
\end{align*}

\begin{lemma}
  Suppose that $\mathcal{L}_a$ satisfies a chain rule and that $\mathcal{L}_a'(\mu)=0$. Let $\S_\mu(\rho)$ be
  the relative entropy between $\rho$ and the invariant measure $\mu$. Then
  \begin{equation}
  \label{eq:exL}
    \mathscr{L}(\rho,0)=\mathscr{L}_s(\rho,-\mathcal{L}_a'(\rho))=\frac{1}{2}\Psi_s(\rho,-\mathcal{L}_a'(\rho))+\frac{1}{2}\Psi^*_s(\rho,d\S_\mu(\rho)).
  \end{equation}
\end{lemma}

\begin{proof}
  Since $\mathcal{L}_a$ satisfies the chain rule, we can simplify $\mathscr{H}_a$ as
\begin{equation*}
  \mathscr{H}_a(\rho;\xi)=\big\langle e^{-\xi} \mathcal{L}_a e^{\xi}\big\rangle_\rho
  =\big\langle\mathcal{L}_a (\xi)\big\rangle_\rho=\langle \mathcal{L}_a'(\rho),\xi\rangle.
\end{equation*}
Since $\mathcal{L}_s$ is symmetric in $L^2_\mu$, $\mathscr{H}_s$ is symmetric around
$d\S_{\mu}(\rho)$~\cite{Mielke2014,DuongOttobre2021}. Therefore, the above decomposition
\begin{equation*}
  \mathscr{H}(\rho,\xi)=\mathscr{H}_a(\rho;\xi)+\mathscr{H}_s(\rho,\xi)
  =\langle \mathcal{L}_a'(\rho),\xi\rangle+\mathscr{H}_s(\rho,\xi).
\end{equation*}
Furthermore, we have
\begin{align*}
  \langle d\S_\mu(\rho), \mathcal{L}_a'(\rho)\rangle&= \langle \log\frac{\rho}{\mu},\mathcal{L}_a'(\rho)\rangle=\langle \rho,\mathcal{L}_a\big(\log\frac{\rho}{\mu}\big)\rangle=\langle \rho, \frac{\mu}{\rho}\mathcal{L}_a(\frac{\rho}{\mu})\rangle
  \\&=\langle 1, \mu\mathcal{L}_a(\frac{\rho}{\mu})\rangle=\langle \frac{\rho}{\mu},\mathcal{L}_a'(\mu)\rangle=0.
\end{align*}
Hence, the above decomposition satisfies the assumptions in Lemma~\ref{lem: H1 linear} with
$\W(\rho)=\mathcal{L}_a'(\rho)$, $\mathscr{H}_2=\mathscr{H}_s$ and $\S=\S_\mu$; therefore, the statement of this
lemma follows from Lemma~\ref{lem: H1 linear}.
\end{proof}

This lemma applies to the non-reversible diffusion process described above; we obtain the splitting 
\begin{multline*}
  \mathscr{L}(\rho,0)=\frac{1}{2}\Psi_s(\rho,-\mathcal{L}_a'(\rho))+\frac{1}{2}\Psi^*_s(\rho,d\S_\mu(\rho))= \\
  \frac{1}{2}\Psi_s(\rho,-\div(\rho D\nabla\log\mu-b\rho))+\frac{1}{2}\Psi^*_s(\rho,d\S_\mu(\rho)),
\end{multline*}
where $\S_\mu$ is the relative entropy.  
This example is chosen for illustrative purposes only; as the dissipation potential $\Psi^\star$ and hence its Legendre dual are quadratic, splittings can be analysed using polarisation identities as mentioned at the beginning of Section~\ref{sec:splitt-hamilt}. Yet, although the splittings developed in this paper are developed with nonquadratic dissipation potentials in mind, the results apply in the quadratic case as well. The interpretation of the result above is as follows: \eqref{eq:exL} is a statement about the stationary (equilibrium) state, $j=0$. The result says that the functional depends there only on the symmetric part $\Psi_s^\star$; there is no contribution from the asymmetric part. So we see that the addition of the asymmetric part does not change the steady state. With additional arguments, it can be shown that away from $j=0$, the (Donsker-Varadhan) rate functional associated with~\eqref{eq:exL} increases when the non-reversible component is present~\cite{ReyBellet2015, Rey-Bellet2016}. As discussed at the beginning of Section~\ref{sec:splitt-hamilt}, this can be interpreted as the asymmetric process converging faster (or at least equally as fast) as the symmetric process.   

 
\section{Conclusion}

In Section~\ref{sec:isosurface}, we have shown that the non-uniqueness of the iso-dissipation force can be
explained through different notions of duality. In Section~\ref{sec:HJ-for-MC}, we have introduced different
splittings of Hamiltonians associated to Markov processes through large deviation principles, and given an
application to a diffusion process on the Fokker-Planck level.

It is natural to compare these decompositions to decompositions in terms of the generator; this is area of
future research. Similar, the application of these splittings for different non-reversible processes remains to
be investigated.

\section*{Acknowledgment}
The research of MHD was supported by EPSRC grants  EP/V038516/1  and  EP/W008041/1.

\bibliographystyle{alpha}
\bibliography{refs}

\begin{thebibliography}{KLMP20}

\bibitem[AT96]{Attouch1996}
H.~Attouch and M.~Théra.
\newblock A general duality principle for the sum of two operators.
\newblock {\em Journal of Convex Analysis}, 3(1):1--24, 1996.

\bibitem[BFG15]{Bertini2015}
Lorenzo Bertini, Alessandra Faggionato, and Davide Gabrielli.
\newblock Flows, currents, and cycles for markov chains: Large deviation
  asymptotics.
\newblock {\em Stochastic Processes and their Applications}, 125(7):2786--2819,
  2015.

\bibitem[DO21]{DuongOttobre2021}
M.~H. Duong and M.~Ottobre.
\newblock Non-reversible processes: Generic, hypocoercivity and fluctuations,
  2021.

\bibitem[DS20]{Dechant2020a}
A.~Dechant and S.~Sasa.
\newblock Fluctuation-response inequality out of equilibrium.
\newblock {\em Proceedings of the National Academy of Sciences},
  117(12):6430--6436, 2020.

\bibitem[DSI22a]{Dechant2022a}
A.~Dechant, S.~Sasa, and S.~Ito.
\newblock Geometric decomposition of entropy production in out-of-equilibrium
  systems.
\newblock {\em Phys. Rev. Research}, 4:L012034, Mar 2022.

\bibitem[DSI22b]{Dechant2022b}
A.~Dechant, S.~Sasa, and S.~Ito.
\newblock Geometric decomposition of entropy production into excess,
  housekeeping, and coupling parts.
\newblock {\em Phys. Rev. E}, 106:024125, Aug 2022.

\bibitem[DV75]{DonskerVaradhan1975a}
M.~D. Donsker and S.~R.~S. Varadhan.
\newblock Asymptotic evaluation of certain {M}arkov process expectations for
  large time, {I}.
\newblock {\em Communications on Pure and Applied Mathematics}, 28(1):1--47,
  1975.

\bibitem[GO97]{GrmelaOttinger1997}
Miroslav Grmela and Hans~Christian \"Ottinger.
\newblock Dynamics and thermodynamics of complex fluids. i. development of a
  general formalism.
\newblock {\em Phys. Rev. E}, 56:6620--6632, 1997.

\bibitem[JKZ17]{Jack2017a}
Robert~L. Jack, Marcus Kaiser, and Johannes Zimmer.
\newblock Symmetries and geometrical properties of dynamical fluctuations in
  molecular dynamics.
\newblock {\em Entropy}, 19(10):562, 2017.

\bibitem[KJZ17]{Kaiser2017}
Marcus Kaiser, Robert~L. Jack, and Johannes Zimmer.
\newblock Acceleration of convergence to equilibrium in markov chains by
  breaking detailed balance.
\newblock {\em Journal of Statistical Physics}, 168(2):259--287, Jul 2017.

\bibitem[KJZ18]{Kaiser2018a}
Marcus Kaiser, Robert~L. Jack, and Johannes Zimmer.
\newblock Canonical structure and orthogonality of forces and currents in
  irreversible markov chains.
\newblock {\em Journal of Statistical Physics}, 170(6):1019--1050, Mar 2018.

\bibitem[KLKS22]{Kobayashi2022}
T.~J. Kobayashi, D.~Loutchko, A.~Kamimura, and Y.~Sughiyama.
\newblock Hessian geometry of nonequilibrium chemical reaction networks and
  entropy production decompositions.
\newblock {\em Phys. Rev. Research}, 4:033208, Sep 2022.

\bibitem[KLMP20]{Kraaij2020}
Richard~C. Kraaij, Alexandre Lazarescu, Christian Maes, and Mark Peletier.
\newblock Fluctuation symmetry leads to generic equations with non-quadratic
  dissipation.
\newblock {\em Stochastic Processes and their Applications}, 130(1):139--170,
  2020.

\bibitem[MPR14]{Mielke2014}
A.~Mielke, M.~A. Peletier, and D.~R.~M. Renger.
\newblock On the relation between gradient flows and the large-deviation
  principle, with applications to markov chains and diffusion.
\newblock {\em Potential Analysis}, 41(4):1293--1327, Nov 2014.

\bibitem[PRS21]{PattersonRengerSharma2022}
R.~I.~A. Patterson, D.~R.~M. Renger, and U.~Sharma.
\newblock Variational structures beyond gradient flows: a macroscopic
  fluctuation-theory perspective, 2021.

\bibitem[RBS15]{ReyBellet2015}
Luc Rey-Bellet and Konstantinos Spiliopoulos.
\newblock Irreversible {L}angevin samplers and variance reduction: a large
  deviations approach.
\newblock {\em Nonlinearity}, 28(7):2081--2103, 2015.

\bibitem[RBS16]{Rey-Bellet2016}
L.~Rey-Bellet and K.~Spiliopoulos.
\newblock Improving the convergence of reversible samplers.
\newblock {\em Journal of Statistical Physics}, 164(3):472--494, Aug 2016.

\bibitem[RZ21]{RengerZimmer2021}
D.~R.~M. Renger and J.~Zimmer.
\newblock Orthogonality of fluxes in general nonlinear reaction networks.
\newblock {\em Discrete and Continuous Dynamical Systems - S}, 14(1):205--217,
  2021.

\bibitem[Sch76]{Schnakenberg1976a}
J{\"u}rgen Schnakenberg.
\newblock Network theory of microscopic and macroscopic behavior of master
  equation systems.
\newblock {\em Reviews of Modern physics}, 48(4):571, 1976.

\end{thebibliography}

\end{document}